\documentclass[12pt]{amsart}
\usepackage{txfonts}
\usepackage{subfig}
\usepackage{amscd,amssymb,mathabx}
\usepackage[arrow,matrix,graph,frame,poly,arc,tips]{xy}
\usepackage{graphicx,calrsfs}
\usepackage[colorlinks=false,plainpages,urlcolor=black,linkcolor=black]{hyperref}

\PassOptionsToPackage{usenames,dvipsnames}{xcolor}
\usepackage{subfig}
\usepackage{tikz}
\usetikzlibrary{decorations.markings}
\usetikzlibrary{shapes}
\usetikzlibrary{backgrounds}
\usetikzlibrary{calc}
\usepackage{mdwlist}
\usepackage{multirow}
\usepackage{enumerate}
\usepackage[shortlabels]{enumitem}
\usepackage{verbatim}

\title{%
Unsmoothable group actions on compact one--manifolds}

\author[H. Baik]{Hyungryul Baik}
\address{Mathematisches Institut, Universt\"at Bonn, 53115 Bonn, Germany}
\email{baik@math.uni-bonn.de}
\urladdr{http://www.math.uni-bonn.de/people/baik/}

\author[S. Kim]{Sang-hyun Kim}
\address{Department of Mathematical Sciences, Seoul National University, Seoul, Korea}
\email{s.kim@snu.ac.kr}
\urladdr{http://cayley.kr}

\author[T. Koberda]{Thomas Koberda}
\address{Department of Mathematics, University of Virginia, Charlottesville, VA 22904-4137, USA}
\email{thomas.koberda@gmail.com}
\urladdr{http://faculty.virginia.edu/Koberda}

\usepackage{enumerate}
\makeatletter
\let\@@enum@org\@@enum@
\def\@@enum@[#1]{\@@enum@org[\normalfont #1]}
\makeatother

\newtheorem{thm}{Theorem}[section]
\newtheorem{lem}[thm]{Lemma}

\newtheorem{cor}[thm]{Corollary}
\newtheorem{prop}[thm]{Proposition}

\newtheorem{que}{Question}

\theoremstyle{definition}
\newtheorem{defn}[thm]{Definition}

\newtheorem{rem}[thm]{Remark}

\newtheorem*{claim*}{Claim}

\newcommand\ba{\begin{align*}}
\newcommand\ea{\end{align*}}
\newcommand\be{\begin{enumerate}}
\newcommand\ee{\end{enumerate}}
\newcommand\bp{\begin{proof}}
\newcommand\ep{\end{proof}}
\newcommand\bpp{\begin{prop}}
\newcommand\epp{\end{prop}}
\newcommand\bpb{\begin{prob}}
\newcommand\epb{\end{prob}}
\newcommand\bd{\begin{defn}}
\newcommand\ed{\end{defn}}
\newcommand\bh{\begin{hint}}
\newcommand\eh{\end{hint}}

\newcommand\N{\mathbb{N}}

\newcommand\Z{\mathbb{Z}}

\newcommand\JJ{\mathcal{J}}

\newcommand\YY{\mathcal{Y}}
\newcommand\ZZ{\mathcal{Z}}

\newcommand\bZ{\mathbb{Z}}

\newcommand\bR{\mathbb{R}}

\newcommand\mC{\mathcal{C}}

\newcommand\form[1]{\langle #1\rangle}

\newcommand\var{\operatorname{var}}

\newcommand\Aut{\operatorname{Aut}}

\newcommand\Out{\operatorname{Out}}

\newcommand\supp{\operatorname{supp}}

\newcommand\Mod{\operatorname{Mod}}

\newcommand\Homeo{\operatorname{Homeo}}

\newcommand\Diffb{\operatorname{Diff}_+^{1+\mathrm{bv}}}
\newcommand\Diff{\operatorname{Diff}}

\DeclareMathOperator\Fix{Fix}
\DeclareMathOperator\Per{Per}
\newcommand\sse{\subseteq}

\newcommand\gam{\Gamma}

\newcommand\co{\colon}

\newcommand\Cbv{C^{1+\mathrm{bv}}}

\renewcommand{\MR}[1]
{\href{http://www.ams.org/mathscinet-getitem?mr=#1}{MR#1}}

\begin{document}

\date{\today}

\begin{abstract}
We show that no finite index subgroup of a sufficiently complicated mapping class group or braid group can act faithfully by $\Cbv$ diffeomorphisms on the circle, which generalizes a result of Farb--Franks, and which parallels a result of Ghys and Burger--Monod concerning differentiable actions of higher rank lattices on the circle. This answers a question of Farb, which has its roots in the work of Nielsen. 
We prove this result by showing that if a right-angled Artin group acts faithfully by $\Cbv$ diffeomorphisms on a compact one--manifold, then its defining graph has no subpath of length three. As a corollary, we also show that no finite index subgroup of $\textrm{Aut}(F_n)$ and $\textrm{Out}(F_n)$ for $n\geq 3$, the Torelli group for genus at least $3$, and of each term of the Johnson filtration for genus at least $5$, can act faithfully by $\Cbv$ diffeomorphisms on a compact one--manifold.
\end{abstract}

\maketitle
\setcounter{tocdepth}{1}

\section{Introduction}

\subsection{Main results}
Let $S=S_{g,n,b}$ be an orientable surface of genus $g$, with $n$ marked points, and with $b$ boundary components, and let $\Mod(S)$ denote its mapping class group. 
We will write \[c(S)=3g-3+n+b\] for the \emph{complexity} of $S$. 
Throughout this paper, we let $M$ be a compact (possibly disconnected) one--manifold, and we let $\Diffb(M)$ be the group of $C^1$ orientation--preserving diffeomorphisms of $M$ whose first derivatives have bounded variation; see Section~\ref{s:onedim} for a more detailed discussion. The following is a classical result of Nielsen:

\begin{thm}[Nielsen~\cite{MR2850125}]\label{t:nielsen}
If $S=S_{g,1,0}$, the group $\Mod(S)$ admits a faithful continuous action on $S^1$ without global fixed points.
\end{thm}

In this article we prove the following result which shows that the action in Theorem~\ref{t:nielsen}, and in fact any action at all even after taking a finite index subgroup, is not smoothable to $C^1$ with bounded variation.

\begin{thm}\label{thm:modc2}
There exists a finite index subgroup $G<\Mod(S)$ and an injective homomorphism from $G$ to $\Diffb(M)$ if and only if $c(S)\leq 1$.
\end{thm}

%\begin{thm}\label{thm:modc2}The mapping class group of $S$ virtually embeds into $\Diffb(M)$ if and only if $c(S)\leq 1$.\end{thm}

Theorem \ref{thm:modc2} generalizes a result of Farb--Franks~\cite{FF2001}, wherein it is shown that if $S=S_{g,n,0}$, where $g\geq 3$ and where $n\in\{0,1\}$, then there is no nontrivial $C^2$ action of $\Mod(S)$ on the circle or the closed interval. As Farb and Franks remark, Ghys established this fact independently in the case of the circle.

Note that a $C^2$ action or a $C^1$ action with Lipschitz derivatives on $M$ is always $\Cbv$. 
For a lower regularity, Parwani~\cite{Parwani2008} proves that every $C^1$ mapping class group action on the circle is trivial if the genus of $S$ is at least six.

%As is standard, we let $B_n$ denote the braid group on $n$ strands. 
We say a group $H$ \emph{embeds} into another group $G$ if there exists an injective homomorphism from $H$ into $G$, and $H$ \emph{virtually embeds} into $G$ if a finite index subgroup of $H$ embeds into $G$.
Theorem \ref{thm:modc2} has the following immediate corollary:

\begin{cor}\label{cor:braid}
%There exists a finite index subgroup $G<B_n$ and an injective homomorphism from $G$ to $\Diffb(M)$ if and only if $n\leq 3$.
The $n$--strand braid group $B_n$ virtually embeds into $\Diffb(M)$ if and only if $n\leq 3$.
\end{cor}

Previous results about the non--existence of smooth mapping class group actions on one--manifolds used particular relations which hold in mapping class groups, especially braid relations. The difficulty with generalizing such arguments is that most relations other than commutation/non--commutation of elements do not persist in finite index subgroups.

We are able to establish Theorem \ref{thm:modc2} by exhibiting a right-angled Artin group which cannot act faithfully by $\Cbv$ diffeomorphisms on a compact one--manifold. 
Recall that if $\gam$ is a finite simplicial graph with the vertex set $V(\gam)$ and the edge set $E(\gam)$, the right-angled Artin group $A(\gam)$ is the finitely presented group 
\[A(\gam)=\langle V(\gam)\mid [v_i,v_j]=1 \textrm{ if and only if } \{v_i,v_j\}\in E(\gam)\rangle.\] 
%This definition is opposite to the more common convention in the literature; see Section~\ref{s:raagmcg}.
We let $P_n$ denote the path on $n$ vertices (see Figure~\ref{f:p4}).

The core of the paper is in proving the following result:

\begin{thm}\label{thm:p4}
The group $A(P_4)$ does not virtually embed into $\Diffb(M)$.
\end{thm}

Note that Theorem \ref{thm:p4} provides a (relatively simple) purely algebraic obstruction for a group to act faithfully on a compact one-manifold by $\Cbv$ diffeomorphisms.
%Note that Theorem \ref{thm:p4} provides a (relatively simple) purely algebraic obstruction for a group to act faithfully on a compact one-manifold by $C^1$ diffeomorphisms whose first derivatives have bounded variation.
Whereas the regularity assumption in Theorem \ref{thm:p4} may appear artificial at first glance, it is actually optimal; in particular, it cannot be weakened to $C^1$. See Theorem \ref{thm:rnilp} below.

Recall that a graph $\Lambda$ is an \emph{induced subgraph} of a graph $\gam$ if $V(\Lambda)\subset V(\gam)$, and whenever $v_1,v_2\in V(\Lambda)$ span an edge of $\gam$, they also span an edge of $\Lambda$. A graph $\gam$ is called \emph{$P_4$--free} if no induced subgraph of $\gam$ is isomorphic to $P_4$.

\begin{cor}\label{cor:p4free}
Suppose that $A(\gam)$ embeds in $\Diffb(M)$. Then $\gam$ is $P_4$--free.
\end{cor}

There are two main steps in the proof of Theorem \ref{thm:p4}. The first step is finding a certain configuration of infinitely many intervals in the supports of the generators of $A(P_4)$. This step constitutes most of the original content of this paper, and builds on the classical results in one--dimensional dynamics described in Section~\ref{s:onedim}. The results of Section~\ref{s:raag-path} give a detailed description of right-angled Artin group actions on one--dimensional manifolds.
Labeling the vertices of $P_4$ as shown in Figure~\ref{f:p4}, we prove the following proposition in Section~\ref{s:raag-path}, which is crucial for the proof of Theorem \ref{thm:p4}:

\begin{prop}\label{p:abcd}
Suppose $M$ is connected. If there is a faithful representation $\phi\co A(P_4)\to \Diffb(M)$
such that $\Fix\phi(v)\neq\varnothing$ for each $v\in V(P_4)$,
then there exist collections of infinitely many intervals 
\[\{I^j_a\co {j\ge1}\}\subseteq\pi_0(\supp \phi(a))\quad\text{and}\quad\{I^j_d\co {j\ge1}\}\subseteq\pi_0(\supp \phi(d))\]
such that the following two conditions hold:
\be[(i)]
\item
$I_v^j\ne I_v^k$ whenever $v\in\{a,d\}$ and $j\ne k$.
\item
For each $j$, we have $I_a^j\ne I_d^j$ and $\phi(cb)I_a^j\cap I_d^j\neq\varnothing$.
\ee 
\end{prop}

In Proposition \ref{p:abcd} above and in what follows, $\pi_0$ denotes the set of connected components.

The second step of the proof of Theorem \ref{thm:p4} consists of derivative estimates of the diffeomorphisms $\{\phi(v)\mid v\in V(P_4)\}$; see Section \ref{s:config-interval}.

In the last section, we will give a generalization of Theorem~\ref{thm:modc2} to finitely generated groups quasi--isometric to mapping class groups.

\subsection{Notes and references}

\subsubsection{Higher rank phenomena for mapping class groups}
Theorem \ref{thm:modc2} is an example of a higher rank phenomenon for mapping class groups. The reader may compare Theorem~\ref{thm:modc2} to the following result of Ghys~\cite{Ghys1999} and Burger--Monod~\cite{BM1999}, which built on work of Witte~\cite{Witte1994}:

\begin{thm}
Let $G$ be a lattice in a simple Lie group of higher rank. Then every continuous action of $G$ on the circle has a finite orbit, and every $C^1$ action of $G$ on the circle factors through a finite group.
\end{thm}

Lattices in higher rank simple Lie groups have Kazhdan's property (T).
Navas~\cite{Navas2002ASENS} proved the following non-smoothability result for all countable groups with property (T).
Here, we use the terminology \emph{$C^{1+\tau}$ diffeomorphism} for a $C^1$ diffeomorphism whose first derivative is H\"older continuous of exponent $\tau>0$:

\begin{thm}\label{thm:navas}
Let $G$ be a countable group with property (T) and $\tau>1/2$.
Then every $C^{1+\tau}$ action of $G$ on the circle factors through a finite group.
\end{thm}

Navas refined Theorem \ref{thm:navas} to groups with relative property (T) in \cite{Navas2005}, and produced a finitely generated, locally indicable group with no faithful differentiable action on the interval in \cite{Navas2010}. In those papers, as in the present one, the main subtleties arise from differentiability.

Braid groups and mapping class groups of genus two surfaces (virtually) admit nontrivial homomorphisms to $\bZ$, so that these groups do admit (non--faithful) actions on the circle without a global fixed point, which can even be taken to be analytic.

In this context, Theorem \ref{thm:modc2} can be viewed as a generalization of known results to ``lattices in mapping class groups", i.e. finite index subgroups of mapping class groups. The reader may note that Theorem \ref{thm:modc2} resolves a question posed by Labourie in his 2002 ICM talk, as well as by Farb (see page 47 of \cite{Farb2006}).

For $C^1$ actions, we can contrast Theorem \ref{thm:p4} with the following result of Farb--Franks~\cite[Theorem 1.6]{FF2003}:
\begin{thm}\label{thm:rnilp}
If $X$ be a connected one--manifold,
then every finitely generated residually torsion--free nilpotent group embeds into the orientation--preserving $C^1$ diffeomorphism group of $X$.
\end{thm}

In particular, right-angled Artin groups, pure braid groups, and Torelli groups of surfaces all admit faithful $C^1$ actions on the circle. In particular, Theorem \ref{thm:p4} cannot be weakened to cover $C^1$ actions of right-angled Artin groups. In light of Parwani's Theorem~\cite{Parwani2008}, one may ask the following:

\begin{que}\label{que:c1}
Let $S$ be a surface with genus at least two. Does $\Mod(S)$ virtually admit a faithful $C^1$ action on a compact one--manifold?
\end{que}

Ivanov conjectured that every finite index subgroup of a higher genus mapping class group has finite abelianization~\cite{Ivanov1992}. Parwani~\cite[Conjecture 1.6]{Parwani2008} noted that Ivanov's conjecture anticipates a 
 negative answer to Question~\ref{que:c1} when the genus is at least four.
%On the other hand, the Thurston Stability Lemma~\cite{Thurston1974Top} states that every nontrivial finitely generated subgroup of $\Diff^1_+[0,1)$ surjects onto $\bZ$. So when combined with the Thurston Stability Lemma, Ivanov's conjecture anticipates a negative answer to Question~\ref{que:c1}.
%The Thurston Stability Theorem~\cite{Thurston1974Top} states that every finitely generated subgroup of $\Diff^1_+( would imply that a finite index subgroup of $\Mod(S)$ surjects onto $\bZ$. Note that Ivanov conjectured the contrary, that is, each finite index subgroup of $\Mod(S)$ has finite abelianization.

\subsubsection{Other group actions}
Theorem \ref{thm:p4} shows that a group which contains $A(P_4)$ cannot act by $\Cbv$ diffeomorphisms on the circle. There are many groups which contain $A(P_4)$, including many subgroups of mapping class groups which do not have finite indices.

We recall the following: let $\Aut(F_n)$ and $\Out(F_n)$ denote the automorphism and the outer automorphism groups of the free group $F_n$ on $n$ generators. The \emph{Torelli group} $\mathcal{I}(S)$ of $S$ is defined to be the kernel of the standard representation $\Mod(S)\to \Aut(H_1(S,\bZ))$. When $S$ has a distinguished marked point, one can write $\mathfrak{J}_k(S)$ for the $k^{th}$ term of the \emph{Johnson filtration}, i.e. $\mathfrak{J}_k(S)$ is the kernel of the natural action of $\Mod(S)$ on the $k^{th}$ universal nilpotent quotient $G_k$ of $\pi_1(S)$. In other words, we set $\gamma_1(\pi_1(S))=\pi_1(S)$ and $\gamma_{k+1}(\pi_1(S))=[\pi_1(S),\gamma_k(\pi_1(S))]$. Then, $G_k=\pi_1(S)/\gamma_k(\pi_1(S))$. Note that with this definition, $\mathfrak{J}_1(S)=\Mod(S)$ and $\mathfrak{J}_2(S)=\mathcal{I}(S)$.

\begin{cor}\label{cor:other groups}
No finite index subgroup of the following groups acts faithfully by $\Cbv$ diffeomorphisms on $M$:
\be
\item
$\Aut(F_n)$ and $\Out(F_n)$ for $n\geq 3$;
\item
$\mathcal{I}(S)$ and $\mathfrak{J}_3(S)$, where $S$ has genus at least $3$;
\item
$\mathfrak{J}_k(S)$, where $S$ has genus at least $5$ if $k > 3$. 
\ee
\end{cor}
It was proved in \cite{BV2003} that $\Aut(F_n)$ does not admit a nontrivial action by circle homeomorphisms, though the Bridson--Vogtmann proof relies essentially on torsion inside $\Aut(F_n)$ and thus does not generalize to finite index subgroups of $\Aut(F_n)$. 
We will give a proof of Corollary \ref{cor:other groups} in Section \ref{s:raagmcg} below.

\subsubsection{Compactness and smooth actions}
Theorem \ref{thm:p4} also underlines a fundamental difference between compact and noncompact one--manifolds, as far as their diffeomorphism groups are concerned. We see in this paper that many right-angled Artin groups do not embed in $\Diffb(M)$ when $M$ is compact, though no such restrictions hold when $M$ is noncompact, as is illustrated by the following result of the authors~\cite{BKK2014}:

\begin{thm}\label{thm:bkk}
Every right-angled Artin group embeds into the orientation--preserving $C^\infty$ diffeomorphism group of $\bR$.
\end{thm}

Actions on noncompact manifolds have applications to mapping class groups of surfaces in their own right:

\begin{que}
Let $n\geq 4$. 
Does $B_n$ virtually admit a faithful $C^\infty$ action on $\bR$?
\end{que}

A negative answer would prove that a braid group on at least four strands does not virtually embed into a right-angled Artin group, by Theorem \ref{thm:bkk}. It has recently been proved that such braid groups do not virtually admit \emph{cocompact} actions on CAT(0) cube complexes~\cite{Haettel2015}, \cite{HJP2015}.

\subsubsection{Structure of finitely generated subgroups of the diffeomorphism group}
Theorems \ref{thm:modc2} and \ref{thm:p4} raise natural questions about the structure of finitely generated subgroups of $\Diffb(M)$, which are known to be quite complicated~\cite{Ghys2001}. Already, we see that the possible right-angled Artin subgroups of $\Diffb(M)$ are rather restricted whenever $M$ is a compact one--manifold. Indeed, Corollary \ref{cor:p4free} states that if $\gam$ is a graph and $A(\gam)$ acts faithfully by $\Cbv$ diffeomorphisms on $M$ then $\gam$ must be $P_4$--free. The class of $P_4$--free graphs is well--understood: it is the smallest class of finite simplicial graphs which contains a single vertex graph, and which is closed under finite disjoint unions and finite joins. See~\cite{CLB1981} and also~\cite{KK2013} for a more detailed discussion of $P_4$--free graphs and the right-angled Artin groups defined by such graphs.

Let us denote by $\Diffb(I,\partial I)$ the group of $\Cbv$ diffeomorphisms of $I$ which are the identity near the boundary of $I$. By concatenating intervals, 
we see that if $G$ and $H$ are finitely generated subgroups of $\Diffb(I,\partial I)$,
then $G\times H$ occurs as a subgroup of $\Diffb(I,\partial I)$.
A much more subtle question is the following:

\begin{que}\label{q:free}
Let $G,H<\Diffb(M)$ be finitely generated subgroups. Does $G*H$ occur as a subgroup of $\Diffb(M)$?
\end{que}

A positive answer to this question would characterize the right-angled Artin subgroups of $\Diffb(M)$ as exactly the ones with $P_4$--free defining graphs. 
Recall that a diffeomorphism of a space is \emph{fully supported} if the fixed point set has empty interior~\cite{FF2001}.
The referee has pointed out the following partial answer to Question~\ref{q:free}: If $G$ and $H$ consist of fully supported diffeomorphisms, then
one can embed $G*H$ into $\Diff_+^2(M)$ by conjugating one of $G$ and $H$ by some diffeomorphism $f$,
and the choice of such an $f$ is dense in $\Diff_+^2(M)$.
%generic (in the sense of Baire) diffeomorphism.
Similar genericity arguments are described in \cite[Proposition 4.5]{Ghys2001} and in \cite[Theorem 2.1]{Rivas2012}.

\subsubsection{Sections for the mapping class group}
Let \[\pi\colon \Homeo_+(S)\to\Mod(S)\] be the projection that takes a homeomorphism of $S$ its mapping class. A well--known result of Markovic~\cite{Markovic2007}, which is closely related to mapping class group actions on the circle, says that $\pi$ does not admit a section whenever $S$ is a closed surface with genus at least six.

\begin{que}
Let $G<\Mod(S)$ be a finite index subgroup. Does there exist a section $s\colon G\to\Diff^1_+(S)$ which splits $\pi$? How about to $\Homeo_+(S)$?
\end{que}

\section{Acknowledgements}
The authors thank N. Kang, S. Lee, J. McCammond, C. McMullen, A. Navas, D. Rolfsen, and M. Sapir for helpful discussions and comments on an earlier draft of this paper. The authors are particularly grateful to B. Farb and A. Wilkinson for helpful comments. The authors thank the hospitality of the Mathematisches Forschungsinstitut Oberwolfach workshop ``New Perspectives on the Interplay between Discrete Groups in Low--Dimensional Topology and Arithmetic Lattices", where part of this research was carried out, and the hospitality of the Tsinghua Sanya International Mathematics Forum, where this research was completed.

The authors are grateful to an anonymous referee who provided many useful comments which significantly improved the exposition.

The first-named author was partially supported by the ERC Grant Nb. 10160104.
The second-named author was supported by Basic Science Research Program through the National Research Foundation of Korea (NRF) funded by the Ministry of Science, ICT \& Future Planning (NRF-2013R1A1A1058646). The second-named author is also supported by Samsung Science and Technology Foundation (SSTF-BA1301-06).

\section{Right-angled Artin groups and finite index subgroups of mapping class groups}\label{s:raagmcg}
For our purposes, right-angled Artin groups are an essential tool for understanding all finite index subgroups of a given mapping class group. In this section, we summarize the relevant facts and reduce the results of this paper to Theorem \ref{thm:p4}.

%Note the difference between our definition of a right-angled Artin group $A(\Gamma)$ and the more common definition \[A(\gam)=\langle V(\gam)\mid [v_i,v_j]=1 \textrm{ if and only if } \{v_i,v_j\}\in E(\gam)\rangle.\]
%More precisely, $A(\gam)$ is isomorphic to $A(\Lambda)$ where $\Lambda$ is the \emph{complement} graph of $\gam$, i.e. $V(\Lambda)=V(\Gamma)$ and $\{v_i,v_j\}\in E(\Lambda)$ if and only if $\{v_i,v_j\}\notin E(\gam)$. Notice that $A(P_4)\cong A(P_4)$, though not in a way which respects the underlying graphs.

\begin{lem}\label{lem:subgroup}
If a right-angled Artin group $A$ embeds into a group $G$,
then $A$ embeds into each finite index subgroup of $G$.
\end{lem}
\begin{proof}
Let $H$ be a finite index subgroup of $G$.
By taking a smaller finite index subgroup if necessary, we may assume $H$ is normal in $G$.
Let $\phi\co A\to G$ be the given injective homomorphism
and  $N=[G:H]$. Write $A=A(\Gamma)$.
Then $K=\langle v^N\mid v\in V(\gam)\rangle\le A$ is isomorphic to $A$
and $\phi(K)\le H$.
\end{proof}

Let $S$ be a surface. Recall that the \emph{curve graph} $\mC(S)$ is a graph whose vertices are isotopy classes of non--peripheral simple closed curves on $S$, and the edge relation is given by disjoint realization on $S$. A graph $\Gamma$ is called an \emph{induced subgraph} of $\mC(S)$
if there exists an injective map of graphs $\iota\colon\gam\to\mC(S)$ which preserves adjacency and non--adjacency.

\begin{thm}[See \cite{Koberda2012}]\label{thm:koberda}
Let $\gam<\mC(S)$ be a finite induced subgraph. Then there exists an injective homomorphism $A(\gam)\to\Mod(S)$.
\end{thm}

Explicitly, let us suppose a graph map $\iota\colon\gam\to\mC(S)$ preserves adjacency and non--adjacency. 
For each $N\in\N$, we get an induced map \[\iota_{*,N}\colon A(\gam)\to\Mod(S)\] by sending $v\mapsto T_{\iota(v)}^N$, where $T_{\iota(v)}$ denotes the Dehn twist about the curve $\iota(v)$. A more detailed version of Theorem \ref{thm:koberda} as proved in \cite{Koberda2012} is that for $N\gg 0$, the homomorphism $\iota_{*,N}$ is injective.

\begin{cor}\label{cor:subgroup}
The group $A(P_4)$ embeds into $\Mod(S)$ whenever $c(S)\ge 2$.
\end{cor}

\begin{proof}[Proof of Theorem \ref{thm:modc2}]
For surfaces $S$ with $c(S)\geq 2$, the conclusion of the theorem follows from Theorem \ref{thm:p4}, combined with Lemma \ref{lem:subgroup} and Corollary \ref{cor:subgroup}. 
%For surfaces with $c(S)< 2$, we have that $\Mod(S)$ is virtually a product of a free groups and cyclic groups and thus admits a $\Cbv$ action on $M$.
For surfaces with $c(S)< 2$, we have that $\Mod(S)$ is virtually a product of a free group and a cyclic group~\cite{CLM2014} and thus virtually admits a $\Cbv$ action on $M$; for example, see~\cite{Grabowski1988}.
\end{proof}

\begin{proof}[Proof of Corollary \ref{cor:other groups}]
By Theorem \ref{thm:p4}, it suffices to find a copy of $A(P_4)$ in each of these groups.
%\be\item

(1)
Note that the surface $S_{1,0,2}$ of genus one with two boundary components contains a \emph{chain} of four simple closed curves, by which we mean a collection of four pairwise non--isotopic essential simple closed curves $\{\gamma_1,\ldots,\gamma_4\}$ in minimal position with the property that $\gamma_i\cap\gamma_{i+1}\neq\varnothing$, but $\gamma_i\cap\gamma_j=\varnothing$ otherwise. See Figure \ref{f:chain} (a) below. By Theorem \ref{thm:koberda}, this realizes $A(P_4)$ in $\Mod(S_{1,0,2})<\Out(F_3)$. Equipping $S_{1,0,2}$ with a marked point realizes $\Mod(S_{1,1,2})$ as a subgroup of $\Aut(F_3)$.

%\item
(2)
A closed surface $S$ of genus at least $3$ contains a chain of four separating closed curves as shown in Figure~\ref{f:chain} (b). The curves are labeled by $a, b, c, d$ so that the intersection graph coincides with the graph shown in Figure~\ref{f:p4} which is a copy of $P_4$. By Theorem \ref{thm:koberda} this furnishes a copy of $A(P_4)$ in $\mathcal{I}(S)$. The reader may note that this actually furnishes $A(P_4)$ in $\mathfrak{J}_3(S)$.

%\item
(3) (Sketch) If $S$ has genus $5$ or more then there exists a chain of four subsurfaces $\{S_1,\ldots,S_4\}$ in $S$, where each $S_i$ is homeomorphic to a genus two surface with one boundary component. A chain of subsurfaces is defined analogously to a chain of simple closed curves, and a chain of two subsurfaces is illustrated in Figure \ref{f:chain} (c) below. Let $k>3$. 

For each $i=1,2,3,4$, we claim
the surface $S_i$ supports a pseudo-Anosov mapping classes $\psi_i$ which lie in the Johnson kernel $\mathfrak{J}_k(S)$. This follows from a well--known fact that every non--central normal subgroup of a mapping class group contains pseudo-Anosov elements (see \cite{Ivanov1992}). One way to see this fact is to appeal to a result of Dahmani--Guirardel--Osin in \cite{DGO2011}, which shows that the mapping class group contains an infinitely generated, purely pseudo-Anosov, normal subgroup $\mathfrak{N}$. If $1\neq \psi\in\mathfrak{J}_k(S)$ is pseudo-Anosov then there is nothing to show. If $\psi$ fails to be pseudo-Anosov then we take $1\neq n\in\mathfrak{N}$, and consider $[n,\psi]$. Since $\psi$ and $n$ will not commute as the centralizer of a pseudo-Anosov mapping class is virtually cyclic, we have that $\phi=[n,\psi]$ is nontrivial. Since $\mathfrak{J}_k(S)$ is normal, $\phi\in\mathfrak{J}_k(S)$. Since $\mathfrak{N}$ is normal, we also have $\phi\in\mathfrak{N}$ and is therefore pseudo-Anosov.

Finally, sufficiently large powers of $\{\psi_1,\ldots,\psi_4\}$ will generate a copy of $A(P_4)<\mathfrak{J}_k(S)$, by \cite[Theorem 1.1]{CLM2012} or \cite[Theorem 1.1]{Koberda2012}.
%\ee
\end{proof}

\begin{figure}[htb!]
\subfloat[(a) A chain of curves]{
\begin{tikzpicture}[scale=.7,rotate=90]
\draw (-.5,0) node {};
\draw (0,-2) -- (9,-2);
\draw (0,2) -- (9,2);

% genus 1
\draw (4.5,.5) edge [out=0,in=90]  (5,0) edge [out=180,in=90] (4,0) ;
\draw (4.5,-.2) edge [out=0,in=-135]  (5.1,.1) edge [out=180,in=-45] (3.9,.1);

% \gamma_1
\draw  (4.5,2) edge [Sepia,thick,out=-75,in=75] (4.5,.5) edge [Sepia,thick,dashed,out=-105,in=105](4.5,.5);

% \gamma_4
\draw  (4.5,-.2) edge [OliveGreen,thick,out=-75,in=75] (4.5,-2) edge [OliveGreen,thick,dashed,out=-105,in=105]  (4.5,-2);

%LL curve
\draw (1,2) edge [red,thick,out=-80, in=180] (4.5,-.7);
\draw [dashed] (1,2) edge [red,thick,out=-100, in=180] (4.5,-1.35);
%LR curve
\draw (8,2) edge [red,thick,out=-100, in=0] (4.5,-.7);
\draw [dashed] (8,2) edge [red,thick,out=-80, in=0] (4.5,-1.35);

%LL curve
\draw [dashed] (1,-2) edge [blue,thick,out=80, in=-180] (4.5,.9);
\draw  (1,-2) edge [blue,thick,out=100, in=-180] (4.5,1.35);
%LR curve
\draw [dashed] (8,-2) edge [blue,thick,out=100, in=0] (4.5,.9);
\draw (8,-2) edge [blue,thick,out=80, in=0] (4.5,1.35);

\node [OliveGreen] at (5,-1.7) {\footnotesize $\gamma_4$};
\node [Sepia] at (5,1.7) {\footnotesize $\gamma_1$};
\node [blue] at (8.5,-1) {\footnotesize $\gamma_2$};
\node [red] at (8.5,1) {\footnotesize $\gamma_3$};
\end{tikzpicture}
}
$\qquad\qquad$
\subfloat[(b) A chain of four separating curves]
{\begin{tikzpicture}[scale=.64,rotate=90]
% upper and lower curve of the surface 
\draw (2,-2) -- (11,-2);
\draw (2,2) edge [out=180,in=180] (2,-2) edge (11,2);

% genus 1,2,3
\draw (2,.5) edge [out=0,in=90]  (2.5,0) edge [out=180,in=90] (1.5,0) ;
\draw (2,-.2) edge [out=0,in=-135]  (2.6,.1) edge [out=180,in=-45] (1.4,.1);
\draw (6,.5) edge [out=0,in=90] (6.5,0) edge [out=180,in=90] (5.5,0);
\draw (6,-.2) edge [out=0,in=-135] (6.6,.1) edge [out=180,in=-45] (5.4,.1);
\draw (10,.5) edge [out=0,in=90] (10.5,0) edge [out=180,in=90] (9.5,0);
\draw (10,-.2) edge [out=0,in=-135] (10.6,.1) edge [out=180,in=-45] (9.4,.1);

% curve d
\draw (6,1.5)  edge  [blue,thick,out=180, in=30] (2.27,.41)  edge [thick,blue,out=0, in=90]  (7.5,0);
\draw (6,-1.5) edge [thick,blue,out=180, in=-30] (2.27,-.16) edge [thick,blue,out=0, in=-90] (7.5,0);
\draw (6,.95) edge [thick,blue,dashed, out=180, in=10] (2.27,.41);
\draw [thick,blue,dashed, out=-10, in=180] (2.27,-.16) edge (6,-.95);
\draw (7,0) edge [thick,blue,dashed,out=90, in=0] (6,.95) edge [thick,blue,dashed,out=-90, in=0] (6,-.95);

% curve a
\draw (9.73,.41) edge [thick,red,out=150, in=0] (6,1.2) edge  [thick,red,dashed, out=170, in=0] (6,.75);
\draw  (9.73,-.16) edge  [thick,red,out=-150, in=0] (6,-1.2) edge [thick,red,dashed, out=-170, in=0] (6,-.5);
\draw (4.5,0) edge  [thick,red,out=-90, in=180] (6,-1.2) edge  [thick,red,out=90, in=180] (6,1.2);
\draw (5,0) edge [thick,red,dashed,out=90, in=180] (6,.75) edge [thick,red,dashed,out=-90, in=180] (6,-.5);

% curve b
\draw  (3.5,2) edge [Sepia,thick,out=-75,in=75] (3.5,-2) edge [Sepia,thick,dashed,out=-105,in=105](3.5,-2);
% curve c
\draw  (8.5,2) edge [OliveGreen,thick,out=-75,in=75] (8.5,-2) edge [OliveGreen,thick,dashed,out=-105,in=105]  (8.5,-2);

\node [Sepia] at (3,-1.5) {\footnotesize $b$};
\node [blue] at (4.8, 1.7) {\footnotesize $d$};
\node [red] at (7.7,1.4) {\footnotesize $a$};
\node [OliveGreen] at (9,-1.5) {\footnotesize $c$};
\end{tikzpicture}
}
$\qquad\qquad$
\subfloat[(c) A chain of subsurfaces]
{
\begin{tikzpicture}[scale=.56,rotate=90]
\draw (1,-2) -- (11.5,-2);
\draw (1,2) edge [out=180,in=180] (1,-2) edge (11.5,2);

% genus 1
\draw (3,1) edge [out=0,in=90]  (3.5,.5) edge [out=180,in=90] (2.5,.5) ;
\draw (3,.3) edge [out=0,in=-135]  (3.6,.6) edge [out=180,in=-45] (2.4,.6);
% genus 2
\draw (6,1) edge [out=0,in=90] (6.5,.5) edge [out=180,in=90] (5.5,.5);
\draw (6,.3) edge [out=0,in=-135] (6.6,.6) edge [out=180,in=-45] (5.4,.6);
% genus 3
\draw (9,1) edge [out=0,in=90] (9.5,.5) edge [out=180,in=90] (8.5,.5);
\draw (9,.3) edge [out=0,in=-135] (9.6,.6) edge [out=180,in=-45] (8.4,.6);

%LL curve
\draw (1,2) edge [Sepia,thick,out=-80, in=180] (4.5,-.7);
\draw [dashed] (1,2) edge [Sepia,thick,out=-100, in=180] (4.5,-1.35);
%LR curve
\draw (8,2) edge [Sepia,thick,out=-100, in=0] (4.5,-.7);
\draw [dashed] (8,2) edge [Sepia,thick,out=-80, in=0] (4.5,-1.35);
\node [Sepia] at (2.3, 1.4) {\footnotesize $S_1$};

%RL curve
\draw (4,2) edge [blue,thick,out=-80, in=180] (7.5,-.7);
\draw [dashed] (4,2) edge [blue,thick,out=-100, in=180] (7.5,-1.35);
%RR curve
\draw (11,2) edge [blue,thick,out=-100, in=0] (7.5,-.7);
\draw [dashed] (11,2) edge [blue,thick,out=-80, in=0] (7.5,-1.35);
\node [blue] at (9.7, 1.4) {\footnotesize $S_2$};
\end{tikzpicture}
}
\caption{Corollary~\ref{cor:other groups}.}
\label{f:chain}
\end{figure}

We also record the following fact, which allows us to reduce to the case where $M$ is a connected manifold:

\begin{lem}\label{lem:connected}
Let $\phi\colon A(P_4)\to G\times H$ be an injective homomorphism, and let $\phi_G$ and $\phi_H$ be the composition of $\phi$ with the projections to $G$ and $H$ respectively. Then one of $\phi_G$ or $\phi_H$ is injective.
\end{lem}
\begin{proof}
Let $K_G$ and $K_H$ be the kernels of $\phi_G$ and $\phi_H$ respectively, which we assume are nontrivial. On the one hand, note that $K_GK_H\cong K_G\times K_H$, so that if $g\in K_G$ and $h\in K_H$ are nontrivial, then $\Z^2\cong\langle g,h\rangle<A(P_4)$. On the other hand, we have that $K_G$ and $K_H$ both contain loxodromic elements, each of which has a cyclic centralizer in $A(P_4)$ (see Lemma 52 of \cite{KK2013b} and Theorem 2.4 of \cite{KMT2014}). This is a contradiction, so that at least one of $K_G$ and $K_H$ is trivial.
\end{proof}

\section{Classical one--dimensional dynamics}\label{s:onedim}
From Sections~\ref{s:onedim} through~\ref{s:config-interval},
we will assume $M$ is connected. In other words, we let
  $M=I=[0,1]$ or $M=S^1=\bR/\bZ$.
For $x,y\in S^1$, the open interval $(x,y)=\{z\in M\co x<z<y\}$
is naturally defined as a subset of the image of $(x,x+1)\sse\bR$.
We let $\Homeo_+(M)$ denote the group of orientation--preserving homeomorphisms of $M$.
For $f\in\Homeo_+(M)$,
we use the notations
$
\Fix f=\{x\in M\co fx=x\}$ and 
$\supp f=M\setminus\Fix f$.
The set $\supp f$
is called
the \emph{support} (or, \emph{open support}) 
of $f$. We will broadly appeal to the following equality, which is trivial to prove: $\supp f=\supp f^{-1}$ for an arbitrary automorphism of an arbitrary set.

We define the set of \emph{periodic points}:

\[\Per f=\{x\in M\co f^px=x\text{ for some }0\ne p\in\bZ\}
=\bigcup_{p=1}^\infty\Fix(f^p).\]
For $Y\sse M$, the set of the connected components of $Y$ is denoted as $\pi_0(Y)$.
%For $M=S^1$, 
%Note that if $f^px=x$ for some $p\ne 0$ and $x\in Z$, then we have $f^p y=y$ for every $y\in\Per f$.
%The following lemma is immediate from the definitions.
For two group elements $a$ and $b$, we let
 $[a,b]=a^{-1}b^{-1}ab$.
\begin{lem}\label{l:commute}
For $f, g\in\Homeo_+(M)$, we have the following.
\be
\item
$\Fix (gfg^{-1})=g\Fix f$, $\supp (gfg^{-1})=g\supp f$
and
$\Per (gfg^{-1})=g\Per f$.
\item\label{l:gx}
Let $Y\sse Z$ be intervals in $M$
such that each component of $\supp f$ is either contained in $Y$ or disjoint from $Z$.
If $g(Y)\sse Y$ and $g(Z)= Z$, then  $gf^{\pm1}g^{-1}(Y)= Y$.
\item
If $f$ and $g$ commute,
then
%[f,g]=1$, then
%$f\Fix g = \Fix g$ and
  $f$ permutes the elements of $\pi_0(\supp g)$.
\ee
\end{lem}

\bp
(1) For $x\in M$, we have
\[
x\in\Fix (gfg^{-1})
\Leftrightarrow
gfg^{-1}x= x
\Leftrightarrow f(g^{-1}x)= g^{-1}x
\Leftrightarrow g^{-1}x\in \Fix f.\]
And the other two assertions follow similarly.
% that $\supp gfg^{-1}=g\supp f$ as well.

(2) The assertion is immediate from 
\[\supp(gf^{\pm1}g^{-1})\cap Z=g\supp f\cap Z\sse g(Y)\sse Y.\]

% $f\supp g=\supp g$ and $f\Fix g=\Fix g$ are obvious from the part (1).

(3)
Let $Y\in\pi_0(\supp g)$. Since $fY$ is connected and contained in $f\supp g=\supp g$,
there exists $Z\in \pi_0(\supp g)$ containing $fY$.
Since $f^{-1}Z$ is connected and contained in $f^{-1}(\supp g)=\supp g$,
we have $Y=f^{-1}Z$ and $fY=Z$.
\ep

Let us say $f$ is \emph{grounded} if $\Fix f\ne\varnothing$.
For example, every map in $\Homeo_+(I)$ is grounded.
We say a group action of $G$ on $M$ is \emph{free}
(or, $G$ acts \emph{freely} on $M$)
if $\Fix g=\varnothing$ for each $g\in G\setminus\{1\}$.

For $f\in \Homeo_+(S^1)$, choose an arbitrary lift $\tilde f\co\bR\to \bR$
and $x\in \bR$.
Then we define the \emph{rotation number} of $f$ as 
$\operatorname{rot} f=
\lim_{n\to\infty} \left(\tilde f^n(x) - x\right)/{n}$,
which is uniquely defined in $\bR/\bZ$ regardless of the choice of the lift $\tilde f$ and $x$.
We have $\operatorname{rot} (f^n)=n\operatorname{rot} (f)$ for each $n\in\bZ$.
It is a standard fact that $\operatorname{rot}(f)=0$ if and only if $f$ is grounded.

\begin{thm}[H\"older's Theorem~\cite{Navas2011}]\label{t:hoelder}
If $G$ is a subgroup of $\Homeo_+(M)$ such that
 $G$ acts freely on $M\setminus\partial M$,
then $G$ is abelian. In this case, if $M=S^1$ then 
the rotation number is an embedding from $G$ to the multiplicative group $S^1$.
%is isomorphic to a subgroup of $\operatorname{SO}(2,\bR)$.
\end{thm}

\begin{thm}[{\cite[Proposition 2.10]{FF2001}}]\label{t:irr-rot}
The centralizer of an irrational rotation in $\Homeo_+(S^1)$ is $\operatorname{SO}(2,\bR)$.
\end{thm}

We denote by $\var(g;M)$ the total variation of a map $g\co M\to\bR$:
%In other words, we have
\[\var(g;M) = \sup\left\{\sum_{i=0}^{n-1}\left| g(a_{i+1})- g(a_i)\right|
\co (a_i\co 0\le i\le n)\text{ is a partition of }M\right\}.\]
In the case $M=S^1$, we require $a_n=a_0$ in the above definition.
Following \cite{Navas2011},
we say a $C^1$ diffeomorphism $f$ on $M$ is $\Cbv$ if $\var(f';M)<\infty$.
We let $\Diffb(M)$ denote
the group of orientation--preserving $\Cbv$ diffeomorphisms of $M$.
Our approach essentially builds on the following two fundamental results on $\Cbv$ diffeomorphisms.
%The following theorem complements Lemma~\ref{l:fp} when $f$ is of $\Cbv$.

\begin{thm}[Denjoy's Theorem {\cite{Denjoy1958}, \cite[Theorem 3.1.1]{Navas2011}}]\label{t:denjoy}
If $f\in\Diffb(S^1)$ and $\Per f=\emptyset$, then $f$ is topologically conjugate to an irrational rotation.
\end{thm}

\begin{thm}[Kopell's Lemma {\cite{Kopell1970}, \cite[Theorem 4.1.1]{Navas2011}}]\label{t:kopell}
Suppose $f,g\in\Diffb(I)$ and $[f,g]=1$.
%\be\item
If $\Fix f\cap (0,1)=\varnothing$ and $g\ne 1$,
then $\Fix g\cap (0,1)=\varnothing$.
%\item Each component of $\Fix f$ is preserved by $g$ and each point in $\partial\Fix f$ is fixed by $g$.\ee
\end{thm}

Farb and Franks proved the following \emph{Abelian Criterion} in {\cite[Lemma 3.2]{FF2001}} for $C^2$ diffeomorphisms. Actually, they stated the theorem in a stronger form for \emph{fully supported} diffeomorphisms, which means that the closure of the support of each diffeomorphism is $I$. Our proof given here follows \cite{FF2001} closely.

\begin{thm}[Abelian Criterion]\label{t:abel}
If $f,g,h\in\Diffb(I)$ satisfy that $\Fix g =\{0,1\}$ and that
$[f,g]=1=[g,h]$, then $[f,h]=1$.
\end{thm}

\bp
Suppose $w\in\form{f,h}$ fixes a point in $(0,1)$. 
Since $[w,g]=1$, Kopell's Lemma implies that
$w=1$. Hence $\form{f,h}$ acts freely on $(0,1)$
and is abelian by H\"older's theorem.
\ep

The \emph{commutation graph} of a subset $S$ of a group $G$
is a graph on the vertex set $S$ such that
 two vertices $u$ and $v$ are joined if and only if they are commuting in $G$. A \emph{complete graph} is a finite graph in which every pair of vertices are joined.

\begin{lem}\label{l:fix2}
Suppose a nonabelian subgroup $G\le\Diffb(S^1)$ is generated by a finite set $V\sse G$ such that the commutation graph of $V$ is connected.
If $V$ consists of infinite order elements, then the rotation number of each $v\in V$ is rational.
\end{lem}

%We denote by $\delta(v)$ the set of adjacent vertices of a vertex $v$ in a graph.
%The centralizer of $g$ in a group $G$ is denoted as $Z_G(g)$.

\bp
Suppose $v_1\in V$ has an irrational rotation number,
and choose a path \[(v_1,\ldots,v_m)\] in the commutation graph of $V$
such that all the vertices are visited at least once.
%Suppose $\Per v=\varnothing$ for some $v\in V$.
By Denjoy's Theorem, there exists $f\in\Homeo_+(S^1)$ such that
 $fv_1f^{-1}$ is an irrational rotation. 
%Let $Z_1=\{u\in V\co [u,v_1]=1\}$.
From Theorem~\ref{t:irr-rot} and the assumption that $v_2$ is of infinite order,
%and the assumption that $V$ consists of torsion-free elements, 
we see that $fv_2f^{-1}$ is also an irrational rotation.
An inductive argument shows that $fvf^{-1}$ is an irrational rotation 
for each $v\in V$. This would imply that $G$ is abelian.\ep

\begin{rem}\label{r:fix2} In Lemma~4.7, one cannot drop the condition that $V$ consists of infinite order elements. For instance, consider the three maps $a,b,c\in\Diff^\infty_+(S^1)$: \[a(x)=x+\frac\pi3, \quad b(x) = x+\frac12, \quad c(x) = x+ \frac{\sin 4\pi x}{8\pi}\mod \bZ.\] It is easy to see that $[a,b]=[b,c]=1$, $[a,c]\ne1$ and that the rotation number of $a$ is irrational.\end{rem}

We write $A\pitchfork B$ if two sets $A$ and $B$ intersect nontrivially.

%For $f\in\Homeo_+(M)$, each point in $\Per f\setminus\Fix f$ is called a \emph{periodic non-fixed point}.

\begin{lem}[Disjointness Condition]\label{l:disjoint}
Components of the supports of two commuting grounded $\Cbv$ diffeomorphisms on $M$
 are either pairwise equal or pairwise disjoint.
\end{lem}
\bp
Suppose $f,g\in\Diffb(M)$ satisfy
$[f,g]=1$,
and $A\pitchfork B$ for some $A=(p,q)\in\pi_0(\supp f)$ and $B=(x,y)\in\pi_0(\supp g)$. 
Assume further that $A\ne B$.
%Without loss of generality, let us assume $A\not\sse B$.
By symmetry, we only need to consider the following two cases.

First, assume $p< x< q\le y$. By replacing $f$ by $f^{-1}$ if necessary,
we may assume $gx=x< fx<q$. Then $fx\in\supp g$ and so, $gfx\ne fx=fgx$. This contradicts the commutativity of $f$ and $g$.
Note that we have also covered the case when we have a circular ordering $p<y\le x<q\le p$.

Second, we assume $p\le x<y\le q$. 
%For each $k\in\bZ$, we have $f^k(B)\sse \supp f^kgf^{-k}\cap A=\supp g\cap A$.
Again, we may assume $f(t)>t$ for $t\in A$.
For each $b\in B$, we have $\lim_{k\to-\infty} f^k(b)=p$
and
$\lim_{k\to\infty} f^k(b)=q$.
Since $gb\in B$, we have
\[
gp = \lim_{k\to-\infty} gf^kb =  \lim_{k\to-\infty} f^k (gb) = p,\
gq = \lim_{k\to\infty} gf^kb =  \lim_{k\to\infty} f^k (gb) = q.\]
%From Lemma~\ref{l:commute}, we see that $S=\{f^k(B)\co k\in\bZ\}\sse\pi_0(\supp g)$. Since $\form{f}$ acts freely on $A$, intervals in $S$ are all distinct, $\lim_{k\to-\infty}f^k(B)=p$ and $\lim_{k\to\infty} f^k(B)=q$. It follows that $gp=p, gq=q$, and hence, $gA=A$. 
Kopell's Lemma for the closure of $A$ yields a contradiction.
\ep

In the following two lemmas, we explain how conjugation can be used to control the supports of homeomorphisms.

\begin{lem}\label{l:fw}
Let $f,g\in\Homeo_+(M)$.
Suppose $Y\sse Z$ are open intervals in $M$ such that each component of $\supp f$ is either contained in $Y$ or disjoint from $Z$.
If $g(Z)= Z$, then there exist $s,t\in\{-1,1\}$ such that
for $u=gf^sg^{-1}$ and $w=uf^tu^{-1}$
we have $w(Y)\sse Y$.
\end{lem}

\iffalse
\begin{lem}\label{l:fw}
If $Y\sse Z$ are open intervals in $M$
and $f,g\in\Homeo_+(M)$ satisfy
$\supp f\cap Z\sse Y$ and $g(Z)= Z$,
then there exist $s,t\in\{-1,1\}$ such that
for $u=gf^sg^{-1}$ and $w=uf^tu^{-1}$
we have $w(Y)\sse Y$.
\end{lem}
\fi

\bp
Write $Y=(p,q)$ and $Z=(P,Q)$.
Let $u=gfg^{-1}$.
We may assume $p\le up$ by replacing $f$ by $f^{-1}$ if necessary.

\emph{Case 1.} $p\le up< uq\le q$.
Let $w=ufu^{-1}$.
Since $u(Y)\sse Y$, Lemma~\ref{l:commute} (\ref{l:gx}) implies
that $w(Y)\sse Y$. %We may also choose $w=uf^{-1}u^{-1}$.

\emph{Case 2.} $p\le up \le q< uq\le Q$.
%We have $uf^\pm u^{-1}p = p$. 
Since $P\le u^{-1}p\le p$, we have $uf^{\pm1}u^{-1}p=uu^{-1}p=p$.
%\[uf^{\pm1}u^{-1}p=uu^{-1}p=p\le up=uf^{\pm1}p\le uf^{\pm1}u^{-1} q\le Q.\]
Hence, if $ufu^{-1}q \le q$, we are done: we choose $w= ufu^{-1}$. 
Otherwise, we have $ufu^{-1}q \ge q$, hence $uf^{-1}u^{-1}q \le q$. In this case, we choose $w = uf^{-1}u^{-1}$. 
In either case, we have $wY\sse Y$.

\emph{Case 3.} $p< q\le up< uq< Q$.
Let $w=ufu^{-1}$.
Since $P\le u^{-1}p< u^{-1}q\le p$,
we have $wp=uu^{-1}p=p$ and  $wq=uu^{-1}q=q$.
Hence, $w(Y)=Y$. 
%We may also choose $w=uf^{-1}u^{-1}$.
\ep

Let $F(x,y)$ be the free group on $\{x,y\}$ and $g\in F(x,y)$.
We say a word $w\in F(x,y)$ is a \emph{successive conjugation of $x$ by $g$} if 
there exist $n\ge1$ and $s_1,s_2,\ldots,s_n\in\{1,-1\}$ 
such that for the sequence of words in $F(x,y)$ defined by
\[w_1=gx^{s_1}g^{-1}, w_{i+1}=w_ix^{s_{i+1}}w_i^{-1}\text{ for }1\le i\le n-1\]
we have $w=w_n$.

\begin{lem}\label{l:fin}
Let $\YY=\{Y_1,Y_2,\ldots\}$ and $\ZZ=\{Z_1,Z_2,\ldots\}$ be two (possibly infinite) collections of open intervals in $M$, such that $Y_i\sse Z_i$ for each $i$ and $Z_i\cap Z_j=\varnothing$ whenever $i\ne j$.
%We assume $\YY$ and $\ZZ$ have the same (possibly infinite) cardinality which is at least $m\ge0$ for some $m$.
If $m\ge0$ and $f,g\in\Homeo_+(M)$ satisfy $\supp f\sse \bigcup\YY$ and $\supp g\sse\bigcup\ZZ$, then
%Suppose for some $m\ge0$, there exists a collection  of disjoint open intervals in $M$  such that $Z_i\cap %\bigcup \YY = Y_i$ for each $i$,  and that the following hold. 
%\begin{itemize} \item If $1\le i\le m$ and $J\sse Y_i$ for some $J\in\pi_0(\supp f)$, then $gJ\sse Z_i$.
%\item If $i>m$ and $J\sse Y_i$ for some $J\in\pi_0(\supp f)$, then $gJ\sse Y_i$. \end{itemize} 
there exists a successive conjugation $w(x,y)$ of $x$ by $y$ in $F(x,y)$ such that for $w=w(f,g)$ and $j\le m$, we have
$wY_j\sse Y_j$ and moreover,
$\supp (wfw^{-1})\sse  \bigcup_{j\le m}Y_i \cup \bigcup_{j>m} Z_i$.
\end{lem}
\bp
For each $Y=Y_i\in\YY$, we define $Z(Y)=Z_i$ so that the group $\form{f,g}$ acts on $Z(Y)$.
Moreover, each component of $\supp f$ is either contained in $Y$ or disjoint from $Z(Y)$.
By Lemma~\ref{l:fw}, we can choose $s_1,t_1\in\{-1,1\}$ such that
for $u_1=u_1(x,y)=yx^{s_1}y^{-1}$
and for $w_1(x,y)=u_1x^{t_1}u_1^{-1}$,
we have $w_1(f,g)(Y_1)\sse Y_1$.

Inductively, suppose we have chosen $ w_i\in F(x,y)\setminus\form{x}$
such that 
$ w_i$ is a successive conjugation of $x$ by $y$
and $ w_i(f,g)(Y_j)\sse Y_j$ for $1\le j\le i<m$.
Since $w_i(f,g)(Z_{i+1})= Z_{i+1}$, 
we can apply Lemma~\ref{l:fw} again. That is, 
for some $s_{i+1},t_{i+1}\in\{-1,1\}$, 
if we let
\[
u_{i+1}(x,y)= w_i(x,y) x^{s_{i+1}}  w_i(x,y)^{-1},
\
w_{i+1}(x,y)=u_i(x,y) x^{t_{i+1}} u_i(x,y)^{-1},\]
then we have $w_{i+1}(f,g)(Y_{i+1})\sse Y_{i+1}$.
By Lemma~\ref{l:commute} (\ref{l:gx}), we see 
$
u_{i+1}(f,g)(Y_j)\sse Y_j$ and
$w_{i+1}(f,g)(Y_j)\sse Y_j$
for $1\le j\le i$.
Inductively, we have $w_m\in\form{x,y}\setminus\form{x}$
such that $w_m(f,g)Y_j\sse Y_j$ for each $1\le j\le m$.
For $j>m$, we have  $w_m(f,g)Y_j\sse w_m(f,g)Z_j\sse Z_j$.
In conclusion, if we put $w=w_m(f,g)$, then
\begin{align*}
\supp wfw^{-1}=w\supp f
&=\bigcup_{j\le m}w(\supp f\cap Y_j)\cup
\bigcup_{j> m}w(\supp f\cap Y_j)
\\
&\sse
\bigcup_{j\le m}Y_j\cup
\bigcup_{j> m}Z_j.\qedhere
\end{align*}
\ep

\section{The right-angled Artin group on a path}\label{s:raag-path}

The main objective of this section is to prove Proposition~\ref{p:abcd}.
Throughout this section, we assume to have a faithful representation $\phi\co A(P_4)\to \Diffb(M)$ such that $\phi(v)$ is grounded for each $v\in V(P_4)$. 
We will continue to assume $M$ is connected.
For brevity, we let each $v\in V(P_4)$ also denote the diffeomorphism $\phi(v)$. 
We label the vertices of $P_4$ as in Figure~\ref{f:p4},
and use the notations $V=V(P_4)=\{a,b,c,d\}$
and $\JJ_v=\pi_0(\supp v)$ for each $v\in V$.
We choose this particular labeling of the vertices so that none of $[a,b],[b,c]$ or $[c,d]$ is trivial.

\begin{figure}[h!]
  \tikzstyle {bv}=[black,draw,shape=circle,fill=black,inner sep=1pt]
\begin{tikzpicture}[thick]
\draw (-1,0) node [bv] {} node [above=.1] {\small $b$} 
-- (0,0)  node [bv] {} node [above=.1] {\small $d$} 
-- (1,0)  node [bv] {} node [above=.1] {\small $a$} 
-- (2,0)  node [bv] {} node [above=.1] {\small $c$};
\end{tikzpicture}%
\caption{The graph $P_4$.}
\label{f:p4}
\end{figure}
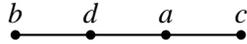

\begin{lem}\label{l:ad}
If $J\in\JJ_a\cap \JJ_d$,
then $\form{a,b,c,d}$ acts on $J$ as an abelian group.
\end{lem}

\bp
For each $v\in\{b,c\}$ and $J'\in\JJ_v$,
the Disjointness Condition  (Lemma~\ref{l:disjoint}) implies that 
either $J\cap J'=\varnothing$ or $J=J'$ holds.
By the Abelian Criterion (Theorem~\ref{t:abel}), we see $\form{a,b,c,d}$ acts as an abelian group on $J$.
\ep

For $f,g\in\Homeo_+(M)$, we define
\begin{align*} 
J(f,g)=
\bigcup\{I_f\cup I_g\co I_f\in\pi_0(\supp f),I_g\in\pi_0(\supp g),I_f\pitchfork I_g,I_f\ne I_g\}\\
\cup
\bigcup\{J\co J\in\pi_0(\supp f)\cap \pi_0(\supp g)\text{ and }\form{f,g}\restriction_J\text{ is nonabelian}\}.
\end{align*}

\begin{lem}\label{l:jfg}
For $f,g\in\Homeo_+(M)$, we have $\supp[f,g]\sse J(f,g)$.
\end{lem}

\bp
If $x\in M\setminus J(f,g)$, then one of the following holds.
\be[(i)]
\item
$x\not\in \supp f\cup \supp g$;
\item
$x\in J\in\pi_0(\supp f)$ such that $J\cap\supp g=\varnothing$;
\item
$x\in J\in\pi_0(\supp g)$ such that $J\cap\supp f=\varnothing$;
\item
$x\in J\in\pi_0(\supp f)\cap\pi_0(\supp g)$ such that $\form{f,g}$ acts on $J$ as an abelian group.
\ee
In each of the cases, we have that $[f,g]x=x$.
\ep

%A function $f$ is said to be \emph{supported} in a set $Y$ if $\supp f\subseteq Y$.

\begin{lem}\label{l:cd}
The following statements hold:
\be
\item
For each $J\in\JJ_c\cup\JJ_d$, we have either $J\subseteq J(c,d)$ or $J\cap J(c,d)=\varnothing$.
%In the former case, one can find a sequence of open intervals $(I_1,\ldots,I_m)$ 
%such that $J=I_k$ for some $k$ and 
%such that the conditions (i) through (v) of Definition~\ref{d:jcde} hold.
\item $\supp a\cap J(c,d)=\varnothing$.
\item An interval in $\JJ_b$ cannot properly contain 
a component of $J(c,d)$.
\ee
\end{lem}

\bp
(1) is immediate. For (2), suppose $I_a\in\JJ_a$ nontrivially intersects  $J\in\JJ_v$
for some $v\in\{c,d\}$ such that $J\sse J(c,d)$. 
By the Disjointness Condition and the Abelian Criterion, we have $I_a=J$ and $\form{a,c,d}$ acts on $J$ as an abelian group.
This contradicts the hypothesis that $J\sse J(c,d)$.

(3)
Each component of $J(c,d)$ contains an interval $J$ in $\JJ_d$.
The Disjointness Condition implies that $J$ cannot be properly contained in $\supp b$.
\ep

\begin{rem}
If $\Fix a$ has an empty interior, then $\supp [c,d]\sse J(c,d)\sse M\setminus \overline{\supp a}=\varnothing$ and $[c,d]=1$ by Lemmas~\ref{l:jfg} and~\ref{l:cd}.  This recovers the \emph{Abelian Criterion} in the form given in \cite[Lemma 3.2]{FF2001}.
\end{rem}

We say a sequence of four intervals $(I_1,I_2,I_3,I_4)$ forms a \emph{chain}
when $I_i\pitchfork I_j$ if and only if $|i-j|=1$.
Figure~\ref{f:coint} illustrates  examples where sequences of four intervals $(I_a,I_b,I_c,I_d)$ form chains.
%Note that we allow some of these sets to coincide.
\begin{figure}[h!]
\subfloat[(a)]{\begin{tikzpicture}[thick,scale=.5]
\draw (-5,0) -- (-1,0);
\draw (-3,0) node [above] {\small $I_a$}; 
\draw (-2,.5) -- (2,.5);
\draw (0,.5) node [above] {\small $I_b$}; 
\draw (1,0) -- (5,0);
\draw (3,0) node [above] {\small $I_c$}; 
\draw (4,.5) -- (8,.5);
\draw (6,.5) node [above] {\small $I_d$}; 
\end{tikzpicture}}
\quad\quad
\subfloat[(b)]{\begin{tikzpicture}[thick,scale=.5]
\draw (-5,0) -- (-1,0);
\draw (-3,0) node [above] {\small $I_a$}; 

\draw (-2,.5) -- (2,.5);
\draw (0,.5) node [above] {\small $I_b$}; 

\draw (1,0) -- (8,0);
\draw (4,0) node [above] {\small $I_c$}; 

\draw (5,.5) -- (7,.5);
\draw (6,.5) node [above] {\small $I_d$}; 
\end{tikzpicture}}
\caption{Chains of intervals.}
\label{f:coint}
\end{figure}
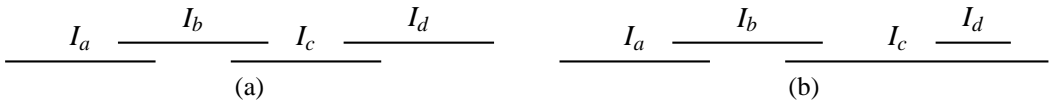

\begin{lem}\label{l:good}
If $I_a\in\JJ_a$ and $I_d\in\JJ_d$ satisfy 
$cb I_a\pitchfork I_d$,
then one of the following holds.
\be
\item $I_a= I_d=cbI_a$ and $\form{a,b,c,d}$ acts on $I_a$ as an abelian group.
\item $I_a\cap I_d=\varnothing$ and there exist $I_b\in\JJ_b$ and $I_c\in\JJ_c$
such that $bI_a\pitchfork I_c$ 
and that $(I_a,I_b,I_c,I_d)$ forms a chain;
in particular, we have $I_d\sse J(c,d)$.
\ee\end{lem}

\bp
The case $I_a=I_d$ is obvious from Lemma~\ref{l:ad}, so let us assume $I_a\cap I_d=\varnothing$.
Choose $x\in I_a$ such that $cbx\in I_d$.
Let us first assume $bx\in I_a$. 
The Disjointness Condition implies that each interval belonging to the family $\JJ_c\cup\JJ_d$ is either
equal to or disjoint from $I_a$. So we have $cbx\in I_a$. 
But this is a contradiction, for $I_a$ and $I_d$ are disjoint.
It follows that $x\in I_a\cap I_b$ and $bx\in I_b\setminus I_a$ for some $I_b\in\JJ_b$.
We have $I_b\cap I_d=\varnothing$ and so, $cbx\not\in I_b$.
This means that there exist $I_c\in\JJ_c$ such that $bx\in I_b\cap I_c$ and
 $cbx\in I_c\cap I_d$.
Since $I_c\pitchfork I_d$ and $I_c\ne I_d$, we see $I_d\sse J(c,d)$. 
\ep

We set $\YY=\pi_0\left(M\setminus\overline{J(c,d)}\right)$.

\begin{lem}\label{l:slippery}
If $I_a\in\JJ_a$ and $I_a\sse Y$ for some $Y\in\YY$,
then $cb I_a$ is contained in the component of 
$M\setminus \left(\bigcup\YY\setminus\{Y\}\right)$ which contains $Y$.
\end{lem}

\bp We have the following:
\begin{claim*}
For each  $Y'\in\YY\setminus\{Y\}$, the intervals $cb I_a$ and $Y'$ are disjoint.\end{claim*}

Suppose the contrary and choose $x\in I_a$ and $Y'\in\YY\setminus\{Y\}$ such that $cb x\in Y'$.
If $bx\in I_a$ then $cb x\in I_a\subseteq Y$ and we have a contradiction.
So $bx\not\in I_a$, and $\{x,bx\}\subseteq I_b$ for some $I_b\in\JJ_b$.
%; note we used the fact that $b$ acts on each component of $\supp b$.
Since $I_a\pitchfork I_b$ and $I_a\ne I_b$, we see $I_b\not\in\JJ_d$. So $I_b\cap \supp d=\varnothing$.
If $cbx \in I_b$, then $I_b\pitchfork Y'$.
From $I_b\pitchfork Y$, we see that $I_b$ would properly contain a component of $J(c,d)$. 
This contradicts Lemma~\ref{l:cd}. 
It follows that $cbx\not\in I_b$, and that we can find $I_c\in\JJ_c$ such that $\{bx, cbx\}\subseteq I_c$.
In particular, $cbx\in I_c\cap Y'$. Since the interval $I_b\cup I_c$ intersects both $Y$ and $Y'$,
a component $J$ of $J(c,d)$ is contained in $I_b\cup I_c$. By Lemma~\ref{l:cd}, the interval $I_b$ cannot properly contain $J$, and so, $I_c\sse J(c,d)$. Since $cbx \in Y'\sse M\setminus J(c,d)$, this is a contradiction.

To complete the proof, let us assume that 
$cbI_a$ is contained in some component $Z$ of $M\setminus\left(\bigcup\YY\setminus\{Y\}\right)$ which does not contain $Y$.
%In particular, $cbI_a\sse \overline{J(c,d)}$.
Fix $x\in I_a$ so that we have $cbx\in Z$.
Each interval joining $x$ and $cbx$ contains an interval in $\YY\setminus\{Y\}$.
%As in the proof of (2) of Lemma \ref{l:good}, 
As in the previous paragraph,
one can see that there exist $I_v\in\JJ_v$ for $v\in\{b,c\}$
such that $bx\in I_b\cap I_c\setminus I_a$ and $cbx\in I_c\setminus I_b$,
and moreover, $I_b\cup I_c$ is an interval containing both $x$ and $cbx$.
By assumption,  $I_b\cup I_c$ contains an interval $Y'\in\YY\setminus\{Y\}$.
This implies that $I_b\cup I_c$ contains at least one components $J$ of $J(c,d)$ such that $J\cap I_c=\varnothing$. This implies that  $I_b$ would properly contain $J$,  and we have a contradiction.
\ep

By Lemma~\ref{l:slippery},
for each $Y\in \YY$ there uniquely exists a minimal open interval $Z(Y)$ contained in \[M\setminus\left(\bigcup\YY\setminus\{Y\}\right)\] such that $Y\sse Z(Y)$
and such that for each $I_a\in \JJ_a$ with $I_a\subseteq Y$, we have $cb I_a\subseteq Z(Y)$.

\begin{lem}\label{l:jzy}
If an open interval $J$ intersects $Z(Y)$ for some $Y\in\YY$,
then either $J\sse Z(Y)$ or $cb I_a\pitchfork J$ for some $I_a\in\JJ_a$ contained in $Y$.
\end{lem}

\bp
Suppose $J\not\sse Z(Y)$ so that one of the endpoints of $Z(Y)$ belongs to $J$.
Then a component of $J\cap Z(Y)$ must intersect $cb I_a$ for some $I_a\in\JJ_a$ satisfying $I_a\sse Y$, by the minimality of $Z(Y)$.
\ep

\begin{lem}\label{l:z-disjoint}
For two distinct intervals $Y$ and $Y'$ in $\YY$, we have $Z(Y)\cap Z(Y')=\varnothing$.
\end{lem}

Each open interval $J\sse M$ can be written as $J=(\inf J,\sup J)$ for some points $\inf J,\sup J\in M$.
\bp[Proof of Lemma \ref{l:z-disjoint}]
Suppose $Z(Y)\pitchfork Z(Y')$. 
By definition, we have $Z(Y)\cap Y'=\varnothing$ and so, $Z(Y')\not\sse Z(Y)$.
By Lemma~\ref{l:jzy}, we can choose $I_a\in\JJ_a$ with $I_a\sse Y$
such that $cbI_a\pitchfork Z(Y')$.
If $cbI_a\not\sse Z(Y')$, then Lemma~\ref{l:jzy} again implies that 
there is $I'_a\in\JJ_a$ with $I'_a\sse Y$ such that $cbI_a\pitchfork cbI'_a$.
But this is impossible since $I_a\cap I'_a=\varnothing$.

It follows that $cbI_a\sse Z(Y')$. By minimality of $Z(Y')$, we can find
$I'_a\in\JJ_a$ such that $I'_a\sse Y$ and such that, up to switching the roles of $Y$ and $Y'$, we have a (possibly circular) ordering of points as following:
\[
\inf I_a<\sup I_a\le \inf (cb I_a')<\sup (cb I_a')\le \inf (cb I_a) <\sup (cb I_a) \le \inf I'_a.\]
This inequality is absurd since $cb$ is order--preserving.
\ep

We let $g=cbab^{-1}c^{-1}$ so that $\supp g= cb \supp a$.
\begin{lem}\label{l:ga}
For each $Y\in\YY$,
we have $aY= Y, dY= Y$ and 
$gZ(Y)= Z(Y)$. Furthermore, we have
$
\supp g
\subseteq \bigcup\{Z(Y)\co Y\in \YY\}$.
\end{lem}

\bp
From $\supp a\cap Y = \bigcup\{J \in\JJ_a\co J\pitchfork Y\}$, we see $aY\sse Y$.
If $I_d\in\JJ_d$ satisfies $I_d\pitchfork Y$ but $I_d\not\sse Y$,
then $I_d\pitchfork J(c,d)$, and hence by Lemma~\ref{l:cd} we have $I_d\sse J(c,d)$. This would be a contradiction, and so, we have $dY=Y$.
For the last two assertions, consider an arbitrary component $cb I_a$ of $\supp g=cb \supp a$.
If $I_a\sse Y'\in\YY$ for some $Y'\ne Y$, then $cb I_a\sse Z(Y')\in M\setminus Z(Y)$ by Lemma~\ref{l:z-disjoint}. If $I_a\sse Y$, then $cb I_a\sse Z(Y)$ by definition. This shows $gZ(Y)\sse Z(Y)$
and $\supp g\sse \bigcup\{Z(Y)\co Y\in \YY\}$.
\ep

Let $\YY_0$ be the collection of intervals $Y\in\YY$ such that
for some $I_a\in\JJ_a$ contained in $Y$
and some $I_d\in\JJ_d$, 
we have $cb I_a\cap I_d\cap (M\setminus Y)\ne \varnothing$;
in this case,
Lemma~\ref{l:good} and the fact $\supp a\cap J(c,d)=\varnothing$ together imply that $I_a\ne I_d$
and $I_d\sse J(c,d)$.

\begin{lem}\label{l:ag}
Let $g$ be as above. Suppose $Y\in \YY$ and $I_d\in\JJ_d$ satisfies $I_d\pitchfork Z(Y)$.
Assume either \be\item $I_d\sse Y$, or \item $Y\not\in \YY_0$.\ee
Then
either $g$ acts on $I_d$ as the identity or $\form{a,b,c,d}$ acts on $I_d$ as an abelian group.
\end{lem}

\bp
Suppose $g$ is not the identity on $I_d$, so that $cb \supp a\pitchfork I_d$.
We can find some $I_a\in\JJ_a$ such that $cb I_a\pitchfork I_d$.

(1) Suppose $I_d\sse Y$.
Since $I_d\sse M\setminus J(c,d)$, 
the case (2) of Lemma~\ref{l:good} does not occur.
It follows that $I_a=I_d=cbI_a$ and that $\form{a,b,c,d}$ is abelian on $I_d$.

(2) Suppose $Y\not\in\YY_0$. If $I_d\sse Z(Y)$, then we have $I_a\sse Y$ by the definition of $Z(Y)$
and Lemma~\ref{l:z-disjoint}.
If $I_d\not\sse Z(Y)$, then Lemma~\ref{l:jzy} implies that we can  find $I'_a\in\JJ_a$
such that $I'_a\sse Y$ and $cb I'_a\pitchfork I_d$. So we may simply assume $I_a\sse Y$.
If $I_d\ne cb I_a$, then Lemma~\ref{l:good} implies that $I_d\sse J(c,d)$ and $Y\in\YY_0$, and this would violate the assumption. So, we have $I_d = cb I_a$. 
By Lemma~\ref{l:good}, the group
$\form{a,b,c,d}$ acts on $I_d$ as an abelian group.
\ep

\begin{lem}\label{l:ysinf}
$\YY_0$ is an infinite collection of intervals.
\end{lem}

\bp
Suppose $\YY=\{Y_1,Y_2,\ldots,Y_m,\ldots\}$ such that
$\YY_0=\{Y_1,Y_2,\ldots,Y_m\}$.
By Lemma~\ref{l:fin}, there is a successive conjugation $w(x,y)\in F(x,y)$ of $x$ by $y$ such that
$w=w(a,g)$ acts on $Y_j$ for $j\le m$ and on $Z(Y_j)$ for $j>m$.
Moreover for $waw^{-1}$, we can require that
 $\supp waw^{-1}\subseteq \bigcup_{j\le m}Y_j \cup\bigcup_{j>m} Z(Y_j)$.

\begin{claim*}
If some $wI_a\pitchfork \supp d$ for some $I_a\in\JJ_a$,
%$J\in\pi_0(\supp waw^{-1})$ intersects $\supp d$,
% $J\pitchfork I_d$ for some $I_d\in\JJ_d$,
then $wI_a\in\JJ_d$ and
$\form{a,b,c,d}$ 
%$\form{waw^{-1},d}$ 
acts on $wI_a$ as an abelian group.
\end{claim*}

Assume $wI_a\pitchfork I_d$ for some $I_d\in\JJ_d$.
Let us first consider the case $I_a\sse Y_j$ for some $j\le m$. 
We have $wI_a\sse Y_j$ and $I_d\pitchfork Y_j$.
By Lemma~\ref{l:cd} (1), we see that $I_d\sse Y_j$.
Then Lemma~\ref{l:ag} (1) implies that $g$ is the identity on $I_d$ or $\form{a,b,c,d}$ is abelian on $I_d$.
If $g$ is the identity on $I_d$, then $wI_d=I_d$ and hence, $I_d=I_a=wI_a$ and
%the restriction of $w$ on $I_d$ is equal to the restriction of $a$ or $a^{-1}$ on $I_d$. So we have 
 $\form{a,b,c,d}$ is abelian on $I_d=wI_a$ 
 by Lemma~\ref{l:ad}.
If $\form{a,b,c,d}$ is abelian on $I_d$, then 
we also have $wI_d=I_d$ and so, $I_d=I_a=wI_a$.

Let us now suppose $I_a\sse Y_j$ for some $j>m$, so that $I_d\pitchfork Z(Y_j)$. 
By Lemma~\ref{l:ag} (2), we see that $I_a\cap \form{a,g}I_d=I_a\cap I_d$.
If $I_a\cap \supp d = \varnothing$, then $I_a\cap I_d=\varnothing$
and $wI_a\cap I_d=w(I_d\cap w^{-1}I_d)=\varnothing$.
So $I_a\in\JJ_d$, and Lemma~\ref{l:ad} implies that $I_a=wI_a=I_d$ and that $\form{a,b,c,d}$ acts on $I_a$ as an abelian group. 
%We can write $w(x,y)=w_m(x,y)=w_{m-1}x^{s_m}w_{m-1}^{-1},\ldots, w_1=yxy^{-1}$ for some $m\ge1$ and $s_2,\ldots,s_m\in\{-1,1\}$. Let us establish the claim by an induction on $m$. (FINISH INDUCTION)

From the claim above and Lemma~\ref{l:jfg}, we see that \[\supp[waw^{-1},d]\sse J(waw^{-1},d)=\varnothing,\]
and so, $[waw^{-1},d]=1$.
On the other hand,
 there exists a successive conjugation $u$ of $a$ by $b$ such that
$[waw^{-1},d]=[cuc^{-1},d]$.
Since $[cuc^{-1},d]$ is a nontrivial reduced word in $A(P_4)$,
we have a contradiction.
%is nontrivial as an element of the right-angled Artin group $A(P_4)$.
\ep
Now Lemma~\ref{l:ysinf} completes the proof of Proposition~\ref{p:abcd}.

\section{The Two Jumps Lemma}\label{s:config-interval}
In this section we prove some quantitative estimates on first derivatives which naturally arise in our setup. We let $M=I$ or $M=S^1$.

\begin{lem}\label{l:fix-accum}
If $f\co M\to M$ is a $C^1$ map
and $x$ is an accumulation point of $\Fix f$,
then $f'(x)=1$.
\end{lem}

\bp
Choose a sequence $(x_n)_{n\ge1}$ converging to $x$ 
such that $x_n\in\Fix f$ and $x_n\ne x$ for each $n$.
Then there exists $y_n$ between $x_n$ and $x$ satisfying 
$f'(y_n)=(f(x_n)-f(x))/(x_n-x)=1$.
We have $1=\lim_{n\to\infty} f'(y_n)=f'(\lim_{n\to\infty} y_n)=f'(x)$.
\ep

The length of an interval $J$ is denoted by $|J|$.
\begin{lem}[Two Jumps Lemma]\label{l:two-jumps}
Let $f,g\co M\to M$ be continuous maps
and $(y_j)_{j\ge 1}$ be an infinite sequence of points in $M$.
For each $j\ge1$, suppose $I_j$ is a closed interval bounded by $f(y_j)$ and $g(y_j)$ such that $y_j$ belongs to the interior of $I_j$
and such that $\lim_{j\to\infty} |I_j|=0$.
For each $j\ge1$,
let $A_j$ and $B_j$ be the closed intervals determined by the following conditions:
\[
I_j=A_j\cup B_j,\  A_j\cap B_j=y_j,\  f(y_j)\in A_j,\  g(y_j)\in B_j.\]
If
$A_j\pitchfork \Fix g$ and $B_j\pitchfork \Fix f$ for each $j\ge 1$, then $f$ or $g$ fails to be $C^1$.
\end{lem}

Figure~\ref{f:fg} (a) illustrates an example of the hypotheses in Lemma~\ref{l:two-jumps}.
%Let $(a_n\co n\ge 1)$ and $(b_n\co n\ge1)$ be infinite sequences of positive real numbers. As is standard, we use the notation $a_n=O(b_n)$ if there exists $C>0$ such that $a_n\le C b_n$ for each $n$. 

\bp[Proof of Lemma \ref{l:two-jumps}]
Suppose $f$ and $g$ are $C^1$,
and choose $s_j\in A_j\cap \Fix g$ and $t_j\in B_j\cap\Fix f$ for each $j\ge1$.
We have a configuration of points as shown in Figure~\ref{f:fg} (a), 
up to choosing a subsequence and reversing the orientation of $M$. So,
%Since $\lim_{j\to\infty}|I_f_j|=0=\lim_{j\to\infty}|I_g_j|$,
 \[f(y_j)\le s_j  < y_j < t_j \le g(y_j).\]
By further passing to a subsequence, we can also assume that $\lim_{j\to\infty}y_j=y$ for some $y\in M$
so that $f(y)=y=g(y)$. Since $(s_j)$ and $(t_j)$ both accumulate at $y$, 
 we see  $f'(y)=1=g'(y)$
from Lemma~\ref{l:fix-accum}.
%We let $\alpha_j=|A_j|$ and $\beta_j=|B_j|$.

Using the Mean Value Theorem, we can find
$u_j\in(s_j,y_j)$ and $v_j\in(y_j,t_j)$ 
satisfying
\begin{align*}
g'(u_j) &= \frac{g(y_j)- s_j}{y_j-s_j} = 1+\frac{|B_j|}{y_j-s_j}
\ge 1+\frac{|B_j|}{|A_j|} \\
f'(v_j) &= \frac{t_j- f(y_j)}{t_j-y_j} = 1+\frac{|A_j|}{t_j-y_j}
\ge 1+\frac{|A_j|}{|B_j|}
\end{align*}
and hence, $f'(v_j)g'(u_j)\ge 4$.
We have a contradiction, since $(u_j)$ and $(v_j)$ both converge to $y$.
\ep

\begin{figure}[h!]
  \tikzstyle {a}=[black,postaction=decorate,decoration={%
    markings,%
    mark=at position 1 with {\arrow[black]{stealth};}    }]
  \tikzstyle {bv}=[black,draw,shape=circle,fill=black,inner sep=1.5pt]
\subfloat[(a)]{
\begin{tikzpicture}[thick,scale=.5]
\draw (-4,0) node [bv] {} node [below]  {\small $f(y_j)$} 
-- (-2,0)  node [bv] {} node [below] {\small $s_j$}
-- (0,0) node [bv] {} node [below] {\small $y_j$} 
-- (2,0)  node [bv] {} node [below] {\small $t_j$}
-- (4,0) node [bv] {} node [below]  {\small $g(y_j)$};
\draw [] (-4,0) -- (-4,1);
\draw [] (4,0) -- (4,1);
\draw [] (0,0) -- (0,1);
\draw (-2,.5) node [above]  {\small $A_j$};
\draw (2,.5) node [above]  {\small $B_j$};
\draw [a,white] (2,.5) -- (4,.5);
\draw [a,white] (2,.5) -- (0,.5);
\draw [dashed] (0,.5) -- (4,.5);
\draw [a,white] (-2,.5) -- (-4,.5);
\draw [a,white] (-2,.5) -- (0,.5);
\draw [dashed] (0,.5) -- (-4,.5);
\end{tikzpicture}%
}\qquad
\subfloat[(b)]{
\begin{tikzpicture}[thick,scale=.7]
\draw (-4,0) -- (2,0);
\draw (-3.8,0) node [above] {\small $I_f^j$}; \draw (-2,.5) -- (4.5,.5);
\draw (4.3,.5) node [above] {\small $I_g^j$}; 
\draw  (-3,0) node [bv] {}; 
\draw  (0,0)  node [bv] {};
\draw [dashed] (-2,.5) -- (-2,-.3);
\draw (-2,-.1) node [below]  {\small $s^j$};
\draw [dashed] (2,.5) -- (2,-.3);
\draw (2,-.1) node [below] {\small $t^j$};
%\draw [] (0,0) -- (0,1.5);
%\draw [] (3,.5) -- (3,1.5);
%\draw [] (-3,0) -- (-3,1.5);
%\draw (-1.5,1) node [above]  {\small $A^j$};
%\draw (1.5,1) node [above]  {\small $B^j$};
%\draw [a,white] (1.5,1) -- (3,1);
%\draw [a,white] (1.5,1) -- (0,1);
%\draw [dashed] (0,1) -- (3,1);
%\draw [a,white] (-1.5,1) -- (-3,1);
%\draw [a,white] (-1.5,1) -- (0,1);
%\draw [dashed] (0,1) -- (-3,1);
\draw (-3,0)   node [below]  {\small $x^j$};
\draw  (3.5,.5)  node [bv] {}  node [below]  {\small $gf(x^j)$};
\draw (0,0)   node [below]  {\small $f(x^j)$};
\end{tikzpicture}%
}
\caption{Lemmas~\ref{l:two-jumps} and~\ref{l:fg}.}
\label{f:fg}
\end{figure}

\begin{lem}\label{l:fg}
Let $f,g$ be orientation--preserving homeomorphisms of $M$.
For each $v\in\{f,g\}$, 
suppose we have an infinite collection of disjoint open intervals
$\{I_v^j\co j\ge 1\}$ such that $v(I_v^j)=I_v^j$ for $j\ge 1$.
Suppose furthermore that for each $j\ge 1$,
there exists some $x^j\in I_f^j\setminus I_g^j$ with the property that
$f(x^j)\in I_g^j$ and $gf(x^j)\not\in I_f^j$.
Then one of $f$ or $g$ fails to be a $C^1$ diffeomorphism.
\end{lem}

\bp
Immediate by applying Lemma~\ref{l:two-jumps} to $f^{-1}$ and $g$ on 
the sequence $\left(f(x^j)\right)_{j\ge1}$. See Figure~\ref{f:fg} (b).
\ep

\section{Proof of the main theorem}
\bp[Proof of Theorem~\ref{thm:p4}]
Suppose $\phi\co A(P_4)\to \Diffb(M)$ is an injective homomorphism,
and we let the vertex set $V=V(P_4)=\{a,b,c,d\}$ as shown in Figure~\ref{f:p4}. 
By raising the generators to suitable powers, we may assume that $\phi(A(P_4))$ preserves each component of $M$.
From Lemma \ref{lem:connected}, we may further assume that $M$ is connected.
By raising to higher powers if necessary, we can also require that
$\phi(v)$ is grounded for each $v$; see Lemma \ref{l:fix2}.
By Proposition \ref{p:abcd} and Lemma~\ref{l:good},
we have an infinite family of distinct intervals $\{I_v^j\co j\ge1\}\sse \pi_0(\supp v)$ for each $v$
such that the hypotheses of Lemma~\ref{l:fg} are satisfied after setting: \[f=b,g=c.\]
We have a contradiction by Lemma~\ref{l:fg}.
\ep

\section{Quasi--isometric rigidity}\label{s:qi}
In this section, we prove the following generalization of Theorem \ref{thm:modc2}:

\begin{thm}\label{thm:qic2}
Let $G$ be a finitely generated group which is quasi--isometric to the mapping class group of a finite type surface with complexity at least two.
Then there is no injective homomorphism $G\to\Diffb(M)$.
\end{thm}

Recall that a \emph{finite type surface} is a closed surface minus finitely many (possibly zero) punctures. 
Let us denote the center $Z(S)=Z(\Mod(S))$.
The proof of Theorem \ref{thm:qic2} relies on the following result:
\begin{thm}[QI rigidity of mapping class groups~\cite{BKMM2012},~\cite{Hamenstadt2005}]
Let $S$ be a finite type surface with $c(S)\geq 2$, and let $G$ be a finitely generated group which is quasi--isometric to $\Mod(S)$. Then after passing to a finite index subgroup of $G$ if necessary, there is a homomorphism $G\to\Mod(S)/Z(S)$ with finite index image and finite kernel.
\end{thm}

% For a finite type surface $S$, we have $|Z(S)|\le 2$.
To prove Theorem \ref{thm:qic2}, let us start with a homomorphism $\phi\colon H_0\to L_0$ such that $[G:H_0]<\infty$, $[\Mod(S)/Z(S):L_0]<\infty$
and such that $K_0=\ker\phi$ is finite.
Let $L_1$ be the preimage of $L_0$ with respect to the projection $\Mod(S)\to\Mod(S)/Z(S)$. See the diagram below.
\[
\xymatrix{
& & G \ar@{-}[d]^{<\infty} & \Mod(S)/Z(S) \ar@{-}[d]^{<\infty}& \Mod(S)\ar@{-}[d]^{<\infty}\ar@{->>}[l]\\
1 \ar[r] & K_0 \ar@{-}[d]\ar[r] & H_0\ar[r]^\phi \ar@{-}[d] & L_0 \ar@{-}[d] & L_1 \ar@{-}[d]\ar@{->>}[l] \\
1 \ar[r] & K\ar[r] & H=\phi^{-1}(L)\ar[r]& L & L\cong A(P_4)\ar[l]_\cong
}
\]
Since $c(S)\ge2$, there exists $L\cong A(P_4)$ inside $L_1$. Since $L$ is torsion-free and $|Z(S)|<\infty$, we have that
$\Mod(S)/Z(S)$ also contains a copy of $L$. We denote this copy again as $L$.
Put $H=\phi^{-1}(L)$ and $K=H\cap K_0$. We obtain an extension \[1\to K\to H\to A(P_4)\to 1.\] 
We wish to show that $H$ cannot act faithfully by $\Cbv$ diffeomorphisms on $M$.

Choosing arbitrary lifts of generators of $A(P_4)$ to $H$, we have that the conjugation action acts by automorphisms of $K$. Since $K$ is finite and acts faithfully on $M$, H\"older's Theorem implies that $K$ is abelian. Moreover, some positive power of each generator of $H$ acts by the identity on $K$, so that we may assume $K$ is central.
Theorem \ref{thm:qic2} now follows immediately from Theorem~\ref{thm:p4} and the following lemma:
\begin{lem}\label{l:nilp}
Let $H$ be a group and $K$ be a finite subgroup of $H$ such that we have a central extension
\[1\to K\to H\to A(P_4)\to 1.\] Then $A(P_4)$ embeds into $H$.
\end{lem}
\begin{proof}
Label $P_4$ as in Figure~\ref{f:p4}, and choose lifts $\{\alpha,\beta,\gamma,\delta\}$ of $\{a,b,c,d\}$ to $H$. We have that $[b,d]=1$ in $A(P_4)$, so that $[\beta,\delta]=q\in K$ in $H$. Suppose $q$ has order $n$. We then compute $[\beta^n,\delta]=q^n=1$ in $H$. Thus, we may replace $\{\alpha,\beta,\gamma,\delta\}$ by appropriate positive powers so that pairs which commute in the projection to $A(P_4)$ will commute in $H$. It follows then that there is a homomorphism $A(P_4)\to H$ which splits the surjection $H\to A(P_4)$, whence it follows that $H$ contains a subgroup isomorphic to $A(P_4)$.
\end{proof}

\begin{rem}
In the proof of Lemma~\ref{l:nilp}, it is immediate that the group $\form{\beta,\delta}$ is nilpotent. 
The Plante--Thurston Theorem~\cite{PT1976}, \cite[Theorem 4.2.3]{Navas2011} states that every nilpotent subgroup of $\Diffb(M)$ is abelian.
Hence with an additional hypothesis that $H\le\Diffb(M)$, we deduce that $[\beta,\delta]=[\delta,\alpha]=[\alpha,\gamma]=1$ and that $\form{\alpha,\beta,\gamma,\delta}\cong A(P_4)$. 
\end{rem}

%\iffalse
\def\cprime{$'$} \def\soft#1{\leavevmode\setbox0=\hbox{h}\dimen7=\ht0\advance
  \dimen7 by-1ex\relax\if t#1\relax\rlap{\raise.6\dimen7
  \hbox{\kern.3ex\char'47}}#1\relax\else\if T#1\relax
  \rlap{\raise.5\dimen7\hbox{\kern1.3ex\char'47}}#1\relax \else\if
  d#1\relax\rlap{\raise.5\dimen7\hbox{\kern.9ex \char'47}}#1\relax\else\if
  D#1\relax\rlap{\raise.5\dimen7 \hbox{\kern1.4ex\char'47}}#1\relax\else\if
  l#1\relax \rlap{\raise.5\dimen7\hbox{\kern.4ex\char'47}}#1\relax \else\if
  L#1\relax\rlap{\raise.5\dimen7\hbox{\kern.7ex
  \char'47}}#1\relax\else\message{accent \string\soft \space #1 not
  defined!}#1\relax\fi\fi\fi\fi\fi\fi}
\providecommand{\bysame}{\leavevmode\hbox to3em{\hrulefill}\thinspace}
\providecommand{\MR}{\relax\ifhmode\unskip\space\fi MR }
\providecommand{\MRhref}[2]{%
  \href{http://www.ams.org/mathscinet-getitem?mr=#1}{#2}
}
\providecommand{\href}[2]{#2}

%\fi
%\bibliographystyle{amsplain}
%\bibliography{ref}

\end{document}